\documentclass{article}
\usepackage{soul}
\usepackage{url}
\usepackage[hidelinks]{hyperref}
\usepackage[utf8]{inputenc}
\usepackage[small]{caption}
\usepackage{graphicx}
\usepackage{amsmath}
\usepackage{booktabs}
\usepackage{algorithm}
\usepackage{algorithmic}
\usepackage{geometry}
\usepackage{dsfont}
\usepackage{bm}
\usepackage{graphicx}
\usepackage{amsthm,mathtools}
\usepackage{tikz}
\newtheorem{example}{Example}
\newtheorem{theorem}{Theorem}

\newtheorem{prop}{Proposition}

\newenvironment{idea}{%
  \proof}{\endproof}

\title{How does information affect asymmetric congestion games?}
\author{
	Charlotte Roman$^1$\footnote{Contact Author} and
	Paolo Turrini$^2$ \\
	\small{$^1$Department of Mathematics, University of Warwick}\\
	\small{$^2$Department of Computer Science, University of Warwick}\\
	\small{\{c.d.roman, p.turrini\}@warwick.ac.uk}
}

\date{}	
\begin{document}

\maketitle

	\begin{abstract}	
	We study traffic networks with multiple origin-destination pairs, relaxing the simplifying assumption of agents having complete knowledge of the network structure. We identify a ubiquitous class of networks, i.e., rings, for which we can safely increase the agents' knowledge without affecting their own overall performance -- known as immunity to Informational Braess' Paradox -- closing a gap in the literature. We also extend our performance measure to include the welfare of all agents, showing that under this measure IBP is a widespread phenomenon and no network is immune to it. 
	\end{abstract}
	
	\section{Introduction}
	With the rising popularity of GPS route-guidance systems, many travellers rely on information about route choice to help them make decisions. It is natural to assume that more information about paths available to a driver would only reduce their expected journey time. However, this is has been shown not to be the case, due to a phenomenon called informational Braess' paradox \cite{Acemoglu}: as presenting new options can change the decision-making of some self-interested agents, this can cause an overall re-routing which makes them worse off.

Congestion games are the standard framework in AI and, more specifically, algorithmic game theory to study the equilibria of traffic flows. These are non-cooperative game where self-interested actors choose sets of available resources, where the cost of each resource depends on its overall usage. In the traffic-specific Wardrop model \cite{Wardrop}, resources form an undirected network in which the players wish to travel between origin and destination nodes, and the cost of each edge often represents expected travel time. Rational players seek to minimise their overall cost by selecting the appropriate path. All players have an origin-destination (OD) pair which corresponds to the nodes they wish to travel between and players with the same OD pair are grouped into populations. 
	
	It is likely that travel costs will change over time such as road improvements or temporary construction works. Braess' paradox (BP) \cite{Braess} occurs when the cost of a resource is reduced but the total cost to the population increases. BP assumes the state of the system is a user equilibrium (or Wardrop equilibrium); where every player has minimised personal costs of travel given the actions of others. BP is based on the assumption that agents have complete information about the network structure. But this assumption is often not met in practical situation, where actors typically have bounded knowledge of their available paths. Recently, information constrained user equilibrium (ICUE) \cite{Acemoglu} has been introduced, where each player minimises their travel costs given their information set (edges of the network). The equilibrium is reached through one-step improvements and gives rise to informational Braess' paradox (IBP) which occurs when a user's cost at ICUE increases as a result of their information set expanding. 
	
ICUE is a significantly more realistic tool to predict traffic flows, but the study of equilibria in congestion games with incomplete information is still at its infancy. In particular, while the relationship between network structure and immunity to IBP has been characterised for single-population congestion games \cite{Acemoglu}, for the more complex and realistic multi-population variant the exact conditions are unknown. On top of that IBP has only been formulated looking at one group, i.e., the one that acquires new knowledge, but not at the welfare of all agents, which is the more social cost metric used for congestion games in general.

\paragraph{Our contribution.}
In this paper we advance the analysis of congestion games played by heterogenous boundedly rational agents, in two important ways: first, we establish how a ubiquitous class of networks, rings, are immune to IBP, settling a conjecture in \cite{Acemoglu}. Second, we extend the analysis of IBP to take into account the welfare of all agents, rather than only a group of them, showing that under this measure IBP is a widespread phenomenon and no network is immune to it.
Our analysis is an important  first step for the design of road systems that do not penalise the acquisition of new knowledge, while establishing that such events do have consequences for the system as a whole.
	
\paragraph{Related Literature.}

	Since the seminal work by Braess \cite{Braess}, the topic of routing paradoxes has been a topic extensively studied \cite{Murchland, Zhao}. Congestion games appear often in transportation systems \cite{Fisk,Pas,Yao,Zhao} but they also apply to a wide variety of real-world systems, with important applications to telecommunication \cite{Orda} and electrical networks \cite{rough}. 
	
	 Pas and Principio \cite{Pas} classified demand constraints and linear cost functions that cause Braess' paradox in traffic games on Braess-like networks. Milchtaich \cite{Milchtaich} singled out the topological conditions for an undirected network to be immune to Braess' paradox for a single population. More generally, Epstein et al. \cite{Epstein} considered topologies in which every Nash equilibrium is socially optimal. In asymmetric games, conditions for network immunity to Braess' paradox was found by Chen et al. \cite{Chen}. More recently, Fujishige et al. \cite{Fujishige} proved the sufficient combinatorial property of strategy spaces for nonatomic congestion games for which Braess' paradox cannot occur, to be those with matroid bases. 
		
	Player heterogeneity has been a topic of much interest,  e.g., through player-specific resource cost functions \cite{Milchtaich2} representing varied preferences \cite{Cole} or uncertainties \cite{artur,Sekar}, while more complex models may consider driver's uncertainties over road conditions or demand \cite{Meir}. There is evidence that providing incomplete information to drivers about road capacities may be worse than providing no information at all \cite{Arnott}. Along the same line, Liu et al. \cite{Liu} studied heterogeneity among players regarding the quality of information they receive and how that affects the equilibrium costs. The informational Braess' paradox was posed by Acemoglu et al. \cite{Acemoglu} for differing levels of information about possible paths. They classified network topologies as immune to this when considering a single OD pair, and show examples of networks that have immunity in the asymmetric game.
	
\paragraph{Paper Structure.} Section \ref{sec:preliminaries} presents the preliminary notions and definitions needed for our results. Section \ref{sec:networks} proves that multi-population congestion games on ring networks will not allow travellers acquiring knowledge to be negatively impacted by information. A variant of performance measure, i.e., social cost, is posed in Section \ref{sec:ibpsc}, which induces losses in utility by some agent independently of the network topology. We conclude discussing future research directions.

\section{Preliminaries}\label{sec:preliminaries}
\subsection{Congestion Games}
Let $N=\{ 1,...,n\}$  be a nonempty finite set of agents populations. In each population, we suppose there exists heterogeneity among knowledge of the resources (due to previous experience, use of GPS systems etc.), i.e., $K_i \geq 1$ information types of players in each population $i$.  We refer to a player from population $i$ of type $k$ as $(i,k)$, which we abbreviate $ik$, where the demand for a type, i.e., the traffic rate associated to that population, is $d_{ik}\geq 0$. 
	
Each population has a nonempty finite resource set $E_i$, where information types can restrict such knowledge, i.e., each population-type pair is associated to a {\em known} set $E_{ik} \subseteq E_i$. We assume that each $E_i$ is made of \textit{relevant} resources, i.e., those which are used in at least one strategy, and that strategy sets $S_{ik} \subseteq 2^{E_{ik}}$ only contain resources from their information set and are disjoint for distinct populations. Denote $E$ as the \textit{irredundant} resource set $E = \bigcup_{i \in N} E_i $. Finally, resource cost functions $c_e: \mathds{R}_{ \geq 0} \rightarrow \mathds{R}_{ \geq 0} \cup \{ \infty \}$ such that $e \in E$ are assumed to be continuous, non-decreasing and non-negative. Formally, a \textit{nonatomic congestion game} is defined as a tuple $\mathcal{M} = (N,(K_i),(E_{ik}),(S_{ik}), (c_e)_{e \in E}, (d_{ik}))$, with $i \in N$ and $ik \in K_i$.
		
The outcome of all players of type $(i,k)$ choosing strategies leads to a strategy distribution $\bm{x}^{ik}$ satisfying $\sum_{s_{ik} \in S_{ik}} x^{ik}_{s_{ik}} = d_{ik}$ and $x^{ik}_{s_{ik}} \geq 0, \,  \forall s_{ik} \in S_{ik}$. A strategy distribution or outcome $\bm{x} = (\bm{x}^{ik})_{\{i \in N, \, k \in K_i\}}$ is \textit{feasible} if $\sum_{s_{ik} \in S_{ik} } x^{ik}_{s_{ik}} = d_{ik}, \, \forall i \in N,\, k \in K_i$. For the rest of the paper we focus, without loss of generality, on feasible strategies. 
	
Denote the \textit{load} on $e$ in an outcome $\bm{x}$ to be $f_e(\bm{x}) =\sum_{i \in N} \sum_{s_i \in S_i} \chi_{\{e \in s_i\} } x^i_{s_i}$. In $\bm{x}$, a player from population $i$ receives a cost function $C(s_{ik},\bm{x})=C_{ik}(s_{ik},\bm{x})=\sum_{e \in s_{ik}} c_e(f_e(\bm{x}))$ when selecting strategy $s_{ik} \in S_{ik}$. 
	
An \textit{information constrained user equilibrium} (ICUE) is a strategy distribution $\bm{x}$ such that  every player of every population chooses a strategy with minimum cost: $\forall i \in N, \, k \in K_i$ and strategies $s_{ik}, s'_{ik} \in S_{(i,k)}$ such that $x^i_{s_{ik}}>0$ we have $C_{ik}(s_{ik},\bm{x}) \leq C_{ik}(s'_{ik},\bm{x})$. The \textit{social cost} is the total cost incurred to all players $SC(\bm{x}) =\sum_{i \in N} \sum_{k \in K_i} C_{ik} (\bm{x})d_{ik}$. 

Braess' paradox is a phenomenon that arises when the cost of a resource is strictly decreased yet results in a strict increase in the social cost of the equilibria. This is motivated by example shown on the Wheatstone network in Figure \ref{fig:Braess}. 
	 
    \begin{center}
	    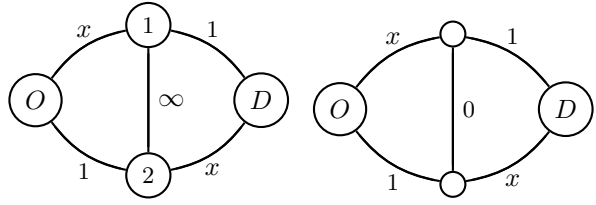
\begin{figure} 
	    	\begin{center}
				\begin{tikzpicture}[shorten >=2pt, thick]
				\node[circle,draw] (A) at (0,0) {$O$};
				\node[circle,draw] (B) at (1.5,1) {\small {1}};
				\node[circle,draw] (C) at (1.5,-1) {\small{2}};
				\node[circle,draw] (D) at (3,0) {$D$};
				\draw[-] (A) to[bend left=20] node[above] {$x$} (B);
				\draw[-] (B) to[bend right=20] node[above] {} (A);
				\draw[-] (B) to[bend left=20] node[above] {\small $1$} (D);
				\draw[-] (D) to[bend right=20] node[above] {} (B);
				\draw[-] (A) to[bend right=20] node[below] {\small $1$} (C);
				\draw[-] (C) to[bend left=20] node[below] {} (A);
				\draw[-] (C) to[bend right=20] node[below] {$x$} (D);
				\draw[-] (D) to[bend left=20] node[below] {} (C);
				\draw[-] (B) to node[right] {$\infty$} (C);
				\draw[-] (C) to node[right] {} (B);   
				\end{tikzpicture}\hspace{0.2cm}
				\begin{tikzpicture}[shorten >=2pt, thick]
				\node[circle,draw] (A) at (0,0) {$O$};
				\node[circle,draw] (B) at (1.5,1) {};
				\node[circle,draw] (C) at (1.5,-1) {};
				\node[circle,draw] (D) at (3,0) {$D$};
				\draw[-] (A) to[bend left=20] node[above] {$x$} (B);
				\draw[-] (B) to[bend right=20] node[above] {} (A);
				\draw[-] (B) to[bend left=20] node[above] {\small $1$} (D);
				\draw[-] (D) to[bend right=20] node[above] {} (B);
				\draw[-] (A) to[bend right=20] node[below] {\small $1$} (C);
				\draw[-] (C) to[bend left=20] node[below] {} (A);
				\draw[-] (C) to[bend right=20] node[below] {$x$} (D);
				\draw[-] (D) to[bend left=20] node[below] {} (C);
				\draw[-] (B) to node[right] {\small $0$} (C);
				\draw[-] (C) to node[right] {} (B);
				\end{tikzpicture}
				\caption{Braess' paradox on the Wheatstone network. When $d=1$, the social cost of travel is $\frac{3}{2}$ before and $2$ after reducing the costs of the middle edge.}
				\label{fig:Braess}
			\end{center}
		\end{figure}
	\end{center}
	 
	 \begin{example} 
    Consider a Wardrop traffic game as depicted in Figure \ref{fig:Braess}. Suppose there is a single population of unit size that wish to travel between nodes $O$ and $D$. In the first instance of cost functions, any UE requires that $\frac{1}{2}$ choose the path $\{ O,1,D \}$ and $\frac{1}{2}$ choose the path $\{ O,2,D \}$ generating a social cost of $\frac{3}{2}$. By reducing the cost of an edge, the resulting UE requires that all players now choose to travel the path $\{ O,1,2,D \}$. This increases the social cost of the equilibrium to $2$, despite only reducing the costs of edges. 
    \end{example}
	 
A set of systems $(E, S_i)_{i \in N}$ admits \textit{Braess' paradox} (BP)  if there are two nonatomic congestion games $\mathcal{M}=(N,E,(S_i)_{i \in N},(c_e)_{e \in E}, (d_i)_{i \in N})$ and $\mathcal{M}'=(N,E,(S_i)_{i \in N},(c'_e)_{e \in E}, (d'_i)_{i \in N})$ where $c'_e(t) \leq c_e(t)$ such that $\forall t \geq 0$ and $d'_i \leq d_i,$ $\forall i \in N$, and two UE $\bm{x}$ and $\bm{x}'$, such that $ SC(\bm{x}) < SC(\bm{x}')$. If no such $\mathcal{M}$ and $\mathcal{M}'$ exist, then we say that the network is \textit{immune} to Braess' paradox. 
	 
\textit{Informational Braess' paradox} (IBP) occurs when one player's type has their information set expanded, without loss of generality type $(1,1)$, and this paradoxically increases their strategy cost. More formally, IBP occurs if there exists expanded information sets $(\tilde{E}_{(i,k)})_{\{ i \in N, \, k \in K_i\}}$ with $E_{(1,1)} \subset \tilde{E}_{(1,1)}$ and $E_{(i,k)} = \tilde{E}_{(i,k)}$ for any $(i,k) \neq (1,1)$ with associated ICUE $\bm{x}$ and $\tilde{\bm{x}}$ where the costs increase for the expanded information player $C_{(1,1)}(\bm{x}) < C_{(1,1)}(\tilde{\bm{x}})$.
	
	\subsection{Graph Theory}

A \textit{simple network} $G=(V,E)$ is an undirected graph with at most one edge between any pair of nodes and no self-loops. A \textit{path} is an ordered collection of edges such that adjacent pairs of edges share a node. If a path visits no node more than once then it is called \textit{acyclic}. A \textit{tree} is a connected simple network that has only acyclic paths. A \textit{ring} is a network such that every node connects to exactly two others, forming a single continuous loop.
	
	A \textit{network congestion game} is played on an undirected network $G=(V,E)$, where the resources are edges and players move between the distinct origin and destination terminal nodes $O_i, D_i \in V$ for any $i \in N$. The strategies of players are paths where no vertex is visited more than once such that the start and end nodes are the associated origin and destination. If a network is \textit{two-terminal} then there is a single origin and destination pair for players to travel between. An \textit{asymmetric} or \textit{multi-population} game is one in which there are multiple $OD$ pairs.
	
	We say that a two-terminal network is \textit{linearly independent} (LI) if each path has at least one edge that does not belong to any other path. A network is \textit{series linearly independent} (SLI) if and only if (i) it comprises a single LI network, or (ii) it is constructed by connecting two SLI networks in series.  For two SLI networks $G_i$ and $G_j$, a \textit{coincident block} is a common LI subnetwork of $G_i$ and $G_j$ with the same set of terminal nodes. An \textit{embedding} is a collection of injective maps from the sets of relevant resources to the irredundant resources. For example, Figure \ref{fig:lattice} shows how the Wheatstone network is embedded in a grid road system. 
	
	\begin{figure} [h]
		\begin{center}
			\begin{tikzpicture}[shorten >=2pt, thick]
			\node[circle,draw] (A) at (2,0) {};
			\node[circle,draw, line width=0.7mm] (B) at (2,1) {$D$};
			\node[circle,draw,line width=0.7mm] (C) at (1,0) {};
			\node[circle,draw] (D) at (0,1) {};
			\node[circle,draw,line width=0.7mm] (E) at (0,0) {$O$};
			\node[circle,draw,line width=0.7mm] (F) at (1,1) {};
			\draw[-] (A) to node {} (B);
			\draw[-] (B) to node {} (A);
			\draw[-] (A) to node {} (C);
			\draw[-] (C) to node {} (A);
			\draw[-] (C) to node {} (F);
			\draw[-] (F) to node {} (C);
			\draw[-] (F) to node {} (B);
			\draw[-] (B) to node {} (F);
			\draw[-] (F) to node {} (D);
			\draw[-] (D) to node {} (F);
			\draw[-] (D) to node {} (E);
			\draw[-] (E) to node {} (D);
			\draw[-] (C) to node {} (E);
			\draw[-] (E) to node {} (C);
			\end{tikzpicture}
			\caption{A Wheatstone network is embedded in any grid.}
			\label{fig:lattice}
		\end{center}
	\end{figure}
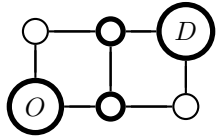
	
	\subsection{Immunity to IBP}	The following theorems from the literature specify the known conditions for IBP to not occur. The first result is for congestion games with only one population of players. For multiple populations, the conditions for immunity to IBP are unknown, with Theorem \ref{theorem:acemloglu} establishing a fairly direct sufficient condition for the multi-population case.
	\begin{theorem} \cite{Acemoglu} \label{theorem:sli}
		A two-terminal network congestion game played on network $G$ is immune to IBP if and only if $G$ is an SLI network.
	\end{theorem}

	\begin{theorem} \cite{Acemoglu} \label{theorem:acemloglu}
		For any asymmetric network congestion game on network $G$, where $\forall i \in N$ $G_i=(V_i,E_i)$ is the relevant network, IBP does not occur if the following hold: \\
		(a) $\forall i \in N$, $G_i$ is SLI \\
		(b) For all distinct $i,j \in N$, either $E_i \cup E_j = \emptyset$, or $E_i \cup E_j$ consists of all coincident blocks of $G_i$ and $G_j$.
	\end{theorem}
	Notice that the conditions from Theorem \ref{theorem:acemloglu} do not depend on information restrictions. It does, however, depend on the maximal information sets of populations which combine to form the relevant network.
	
	\section{Circuit games and IBP}\label{sec:networks}
%% make this better 
	In this section, we formally introduce circuits, most commonly found in matroid theory, for which immunity to IBP is conjectured \cite{Acemoglu}. Consider a set system $(E,\bm{\mathcal{C}})$, where $E$ is the set of resources and $\bm{\mathcal{C}} \subseteq 2^{E}$, with the following axioms:
	 \begin{itemize}
	 	\item $\emptyset \notin \bm{\mathcal{C}}$; 
	 	\item If $\mathcal{C}_1,\mathcal{C}_2 \in \bm{\mathcal{C}}$ and $\mathcal{C}_1 \subseteq \mathcal{C}_2$, then $\mathcal{C}_1=\mathcal{C}_2$,
	 	\item For any two distinct $\mathcal{C}_1,\mathcal{C}_2 \in \bm{\mathcal{C}}$ such that $e \in \mathcal{C}_1 \cap \mathcal{C}_2$, there is a member $\mathcal{C}_3 \in \bm{\mathcal{C}}$ such that $\mathcal{C}_3 \subseteq (\mathcal{C}_1 \cup \mathcal{C}_2)\backslash \{e \}$.
	 \end{itemize} 
 	Then $\bm{\mathcal{C}}$ is a \textit{circuit} over $E$. A  \textit{circuit game} is a congestion game in which every relevant network $G_i=(V_i,E_i)$ $\forall i \in N$ is a circuit.
 	We now establish this useful lemma, whose proof we only sketch due to limitations of space.
	
	\begin{prop} \label{prop:circuit}
		A circuit game cannot exist on a network $G$ that is not simple. 
	\end{prop}
	\begin{proof}
		By contradiction. Suppose we have a network $G$ that is not simple. Choose two nodes where there exist multiple edges between them- $e_1,e_2$. Suppose a population $i$ of players can use $e_1$ in their strategy then $e_2 \in E_i$ since they connect the same nodes. By definition of a circuit game, the relevant network of that population must be a circuit $\mathcal{C}_i$. We must have $e_1, e_2 \in \mathcal{C}_i$. Consider $\mathcal{C}'_i=\{e_1,e_2\}$. It is a dependent set since it is a cycle. Any of its proper subsets are independent, therefore, it is a circuit. By the axioms of circuits, if $\mathcal{C}'_i \subseteq \mathcal{C}_i$ then $\mathcal{C}_i=\mathcal{C}'_i$. So the population $i$ can only travel between the end nodes of $e_1$ and $e_2$. There must exist another population $j$ whose strategies also include $e_1$ and $e_2$ since otherwise, the congestion game of populations $N \backslash \{ i \}$ is an equivalent game (one with the same equilibria). Population $j$ must also have a relevant network that is a circuit and hence must travel on the circuit $\mathcal{C}_j=\{ e_1, e_2 \} = \mathcal{C}_i$. This implies that the populations are indistinct which is a contradiction. Hence, $G$ must be simple.
	\end{proof}
	
	Proposition \ref{prop:circuit}, notice, implies that any ring forms a circuit game. Now we prove the following:
	
	\begin{prop} \label{2circuit}
		Any two-terminal circuit game is immune to IBP.
	\end{prop}
	The proof of Proposition \ref{2circuit} follows from Theorem \ref{theorem:sli} since it can be shown, using Proposition \ref{prop:circuit}, that a circuit game network is SLI. This result confirms the current classification of IBP immunity from Theorem \ref{theorem:acemloglu}.

	When considering multiple origin-destination pairs, the circuit axioms now apply to slightly more complex structures than simple rings. A circuit games can comprise connected rings such that the OD pairs do not allow for traversal between rings. Before we can prove the immunity to IBP for such structures, we pose the more general statement that not all player types can be negatively impacted by a single information expansion. We omit the proof for the general case, but we give the idea behind the proof technique and present the full reasoning for the case of $n=2$.
	
	\begin{prop} \label{prop:nonempty}
Let  $\mathcal{M}$ be an asymmetric circuit game and and let $(\tilde{E}_{(i,k)})_{\{ i \in N, \, k \in K_i\}}$ be expanded information sets such that $E_{(1,1)} \subset \tilde{E}_{(1,1)}$ and $E_{(i,k)} = \tilde{E}_{(i,k)}$ for any $(i,k) \neq (1,1)$, wit associated ICUE $\bm{x}$ and $\tilde{\bm{x}}$. Then there exists at least one player type $(i,k)$ $i \in N, k \in K_i$ such that $C_{(i,k)}(\bm{\tilde{x}}) \leq C_{(i,k)}(\bm{x})$.
	\end{prop}

\begin{idea}
		In a circuit, we can show that each population has at most two strategies. Each population that has only one strategy will not affect the equilibria before and after the information expansion, except the type $(1,1)$ whose strategy set expands. So we can assume that for $n$ populations there are at most $n+1$ types. Since there are $n$ populations with distinct OD pairs, the game must be embedded in a $2n$-edge circuit. Now if we compare the contradiction assumption for each of the types, we know that the strategy the use before the expansion must have strictly less flow on at least one of the edges in the strategy than afterwards. Now if we compare all $n+1$ inequalities we will find that there is always a contradiction since the demands of populations must be non-negative. 
	\end{idea}
	
	The following example shows the  for the case where $n=2$ as displayed in Figure \ref{fig:nonempty}. 
	\begin{center}
		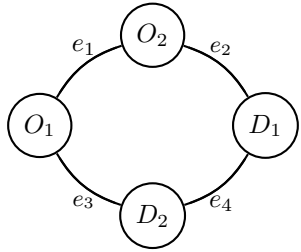
\begin{figure} [h]
			\begin{center}
				\begin{tikzpicture}[shorten >=2pt, thick]
				\node[circle,draw] (A) at (0,0) {$O_1$};
				\node[circle,draw] (B) at (1.5,1.2) {$O_2$};
				\node[circle,draw] (C) at (1.5,-1.2) {$D_2$};
				\node[circle,draw] (D) at (3,0) {$D_1$};
				\draw[-] (A) to[bend left=20] node[above] {$e_1$} (B);
				\draw[-] (B) to[bend right=20] node[above] {} (A);
				\draw[-] (B) to[bend left=20] node[above] {\small $e_2$} (D);
				\draw[-] (D) to[bend right=20] node[above] {} (B);
				\draw[-] (A) to[bend right=20] node[below] {\small $e_3$} (C);
				\draw[-] (C) to[bend left=20] node[below] {} (A);
				\draw[-] (C) to[bend right=20] node[below] {$e_4$} (D);
				\draw[-] (D) to[bend left=20] node[below] {} (C);
				\end{tikzpicture}
				\caption{A circuit with two populations.}
				\label{fig:nonempty}
			\end{center}
		\end{figure}
	\end{center}
	\begin{example}
		%%% Change this so that d12=0 %%%
		Let there be two populations each with distinct $OD$ pairs as shown in Figure \ref{fig:nonempty}. Let the strategy sets of player types: $(1,1)$ before $S_{11}=\{s_1\}$ and after expansion $\tilde{S}_{11}=\{s_1,s_2\}$, type $(1,2)$ has strategies $S_{12}=\{s_1,s_2\}$, and type $(2,2)$ has $S_{22}=\{t_1,t_2\}$. The demands for these populations are $d_{11}>0, \, d_{12} \geq 0, \, d_{22}>0$. In order to reach a contradiction, suppose that there does not exist such a player. Then $\forall i \in \{1,2\}, \, k \in K_i$ we have $C_{(i,k)}(\bm{\tilde{x}}) > C_{(i,k)}(\bm{x})$. 
		
		Consider the feasible strategy distributions $\bm{x}$. Since type $(1,1)$ only has one strategy we have $x^{s_1}_{11} = d_{11}$. In order for every player's cost to increase, we must have that $(1,1)$ strictly prefers to deviate to the other strategy hence $(1,2)$ must choose $s_2$. This gives us $x^{s_2}_{12} = d_{12}$. Now suppose that $(2,2)$ plays $x^{t_1}_{22} = pd_{22}$, $x^{t_2}_{22} = (1-p)d_{22}$. The cost functions for the players are: 
		\[ 
		\begin{array}{lll}C_{11}(\bm{x})= & c_{e_1}(f_1) +  c_{e_2}(f_2)& \\
		C_{12}(\bm{x})= & c_{e_3}(f_3) + c_{e_4}(f_4)& \\ 
		C_{22}(\bm{x})= & \begin{cases}
		c_{e_1}(f_1)+ c_{e_3}(f_3) & p \in [0,1) \\
		c_{e_2}(f_2) + c_{e_4}(f_4) & p \in (0,1]
		\end{cases} 
		\end{array} 
		\]
		where 
		\[ 
		\begin{array}{ll}
		f_1= & d_{11}+(1-p)d_{22} \\
		f_2= & d_{11}+pd_{22} \\ 
		f_3= & d_{12}+(1-p)d_{22} \\
		f_4= & d_{12}+pd_{22} 
		\end{array}
		\]
		Now consider the possible strategy distributions $\tilde{\bm{x}}$. Without loss of generality, assume that both $(1,1)$ and $(1,2)$ have the same strategy distribution. Consider the following strategy distribution for population 1: $\tilde{x}^{s_1}_{1k}= d_{1k}$, $k \in \{1,2\}$. This could only be an ICUE if it was a dominant strategy given the total demand of population 2, which we know is not true given the deviation from $s_1$ in strategy distribution $\bm{x}$. So we must have $\tilde{x}^{s_1}_{1k}=qd_{1k}$ and $\tilde{x}^{s_1}_{1k}=(1-q)d_{1k}$ where $k \in \{1,2\}$ and $q \in [0,1)$. For population 2, let the strategy distribution be $x^{t_1}_{22} = \tilde{p}d_{22}$, $x^{t_2}_{22} = (1-\tilde{p})d_{22}$. The cost functions after information expansion are:
		\[
		\begin{array}{ll}
		C_{11}(\bm{\tilde{x}})= & c_{e_3}(\tilde{f}_3) + c_{e_4}(\tilde{f}_4) \\
		C_{12}(\bm{\tilde{x}})= & c_{e_3}(\tilde{f}_3) + c_{e_4}(\tilde{f}_4) \\
		C_{22}(\bm{\tilde{x}})= & \begin{cases}
		c_{e_1}(\tilde{f}_1) +  c_{e_3}(\tilde{f}_3) & \tilde{p} \in [0,1)\\
		c_{e_2}(\tilde{f}_2) + c_{e_4}(\tilde{f}_4) & \tilde{p} \in (0,1]
		\end{cases} \\
		\end{array}
		\]
		where
		\[ 
		\begin{array}{ll} 
		\tilde{f}_1= & q(d_{11}+d_{12})+(1-\tilde{p})d_{22} \\
		\tilde{f}_2= & q(d_{11}+d_{12})+\tilde{p}d_{22} \\
		\tilde{f}_3= & (1-q)(d_{11}+d_{12})+(1-\tilde{p})d_{22} \\ 
		\tilde{f}_4= & (1-q)(d_{11}+d_{12})+\tilde{p}d_{22} \\
		\end{array}
		\]
		The contradiction assumption
		$C_{11}(\bm{x}) < C_{11}(\bm{\tilde{x}})$ gives us:
		\[ 
		c_{e_1}(f_1) + c_{e_2}(f_2) < c_{e_3}(\tilde{f}_3) + c_{e_4}(\tilde{f}_4) \leq c_{e_1}(\tilde{f}_1) + c_{e_2}(\tilde{f}_2)
		\] 
		Since cost functions are nondecreasing, this implies that we must have $f_1<\tilde{f}_1$ or $f_2<\tilde{f}_2$. In terms of demands: (i) $d_{11}+(1-p)d_{22} < q(d_{11}+d_{12})+(1-\tilde{p})d_{22}$ or (ii) $d_{11}+pd_{22}<q(d_{11}+d_{12})+\tilde{p}d_{22}$. If both these hold, we have $\tilde{p}+\frac{(1-q)d_{11}-qd_{12}}{d_{22}}<p$ and $\tilde{p}+\frac{(1-q)d_{11}-qd_{12}}{d_{22}}>p$ which leads to a contradiction.
		
		The contradiction assumption for player type $(1,2)$, $C_{12}(\bm{x}) < C_{12}(\bm{\tilde{x}})$, leads to
		\[
		c_{e_3}(f_3) + c_{e_4}(f_4) <  c_{e_3}(\tilde{f}_3) + c_{e_4}(\tilde{f}_4)
		\] 
		By nondecreasing cost functions, we have (iii) $d_{12}+(1-p)d_{22}<(1-q)(d_{11}+d_{12})+(1-\tilde{p})d_{22}$ and/or (iv) $d_{12}+pd_{22}<(1-q)(d_{11}+d_{12})+\tilde{p}d_{22}$. 
		But notice that (i) and (iv) contradict each other, in addition (ii) and (iii) contradict each other. So we either have both (i) and (iii) holding true, or the case where (ii) and (iv) hold true. Inequality (iii) cannot hold unless (i) holds and (iv) cannot hold without (ii). So we either need (i) or (ii) to hold. 
		
		Finally, consider $C_{22}(\bm{x}) < C_{22}(\bm{\tilde{x}})$.
		Consider the case where $p, \tilde{p} \in (0,1)$. Then we must have 
		\[ 
		c_{e_1}(f_1) + c_{e_3}(f_3) < c_{e_1}(\tilde{f}_1) + c_{e_3}(\tilde{f}_3) \]
		\[
		c_{e_2}(f_2) + c_{e_4}(f_4) < c_{e_2}(\tilde{f}_2) + c_{e_4}(\tilde{f}_4)
		\] 
		
		Similarly, by nondecreasing cost functions, we must have either (v) $d_{11}+(1-p)d_{22}<q(d_{11}+d_{12})+(1-\tilde{p})d_{22}$ or (vi) $d_{12}+(1-p)d_{22}<(1-q)(d_{11}+d_{12})+(1-\tilde{p})d_{22}$. These are equivalent to (i) and (iii) respectively. It must also be true that at least one of $d_{11}+pd_{22}<(q(d_{11}+d_{12})+\tilde{p}d_{22}$ or $d_{12}+pd_{22}<(1-q)(d_{11}+d_{12})+\tilde{p}d_{22}$ is true. However, these are equivalent to (ii) and (iv) respectively. So this reaches a contradiction.  \footnote{Notice that $d_{12}$ can be reduced from all inequalities so we can omit type $(1,2)$ and get the same result.} 
		
		Now suppose $p=[0,1), \, \tilde{p} \in [0,1)$. Then we have: 
		\[
		c_{e_1}(f_1) + c_{e_3}(f_2) < c_{e_1}(\tilde{f}_1) + c_{e_3}(\tilde{f}_3)
		\]
		So we have that (i) and (iii) hold. This gives us $\frac{(1-q)d_{11}-qd_{12}}{d_{22}}+\tilde{p} < p$ which means $p>0$ since $d_{11},d_{22}>0$.
		So $\tilde{p}=0$ otherwise we have the case as above. Now this implies that 
		\[  
		c_{e_2}(f_2) + c_{e_4}(f_4) < c_{e_2}(\tilde{f}_2) +  c_{e_4}(\tilde{f}_4)  
		\] 
		This again will lead to a contradiction.
		
		Now suppose $p=(0,1],\, \tilde{p} \in (0,1]$.
		\[ 
		c_{e_2}(f_2) + c_{e_4}(f_4) < c_{e_2}(\tilde{f}_2) + c_{e_4}(\tilde{f}_4)
		\] 
		So we have that (ii) and (iv) hold. This gives us $\frac{(1-q)d_{12}-qd_{11}}{d_{22}}+p < \tilde{p}$, using similar reasoning as above we see that we must have $p<1$ and $\tilde{p}=1$. But this tells us:
		\[ 
		c_{e_1}(f_1) + c_{e_3}(f_3) < c_{e_1}(\tilde{f}_1) + c_{e_3}(\tilde{f}_3)
		\]
		Hence, we reach a contradiction. 
		
		Now suppose $p=[0,1),\, \tilde{p}=(0,1]$. Then we see that 
		\[
		c_{e_1}(f_1) + c_{e_3}(f_3) <
		c_{e_1}(\tilde{f}_1) + c_{e_3}(\tilde{f}_3)
		\] 
		So we must have that (i) and (iii) hold so we have $\frac{(1-q)d_{11}-qd_{12}}{d_{22}}+\tilde{p} < p$. Hence, we must have $p, \tilde{p} \in (0,1)$ which we have already shown to be contradictory.
		
		Now suppose $p=(0,1],\, \tilde{p}=[0,1)$.
		\[ 
		c_{e_2}(d_{11}+pd_{22}) + c_{e_4}(d_{12}+pd_{22}) < c_{e_2}(\tilde{f}_2) + c_{e_4}(\tilde{f}_4)
		\] 
		This implies $\frac{(1-q)d_{11}-qd_{12}}{d_{22}}+p < \tilde{p}$ must hold which means that we have $p, \tilde{p} \in (0,1)$ which leads to a contradiction.

\end{example}

Now that we have shown that a circuit game will not increase all player's costs from information distribution simultaneously, we can prove that information cannot harm the player who receives it.
	
	\begin{theorem} \label{theorem:circuit}
		Any circuit game is immune to IBP.
	\end{theorem}
	\begin{proof}
		By the definition of IBP, there exists an information type whose information set expands. Assume without loss of generality that type $(1,1)$ are those players with expanded information sets. To reach a contradiction, assume that $C_{(1,1)}(\bm{\tilde{x}})>C_{(1,1)}(\bm{x})$ where $\bm{x}$ and $\bm{\tilde{x}}$ are the ICUEs reached before and after the information set of type $(1,1)$ is expanded respectively. 
		
		Since each player's relevant network is a circuit, the possible paths must be divided into the two directions that one can travel around the circuit. By Proposition \ref{prop:circuit}, there is only a single edge that can connect any two nodes, so each population must have two linearly independent paths available to them. Player type $(1,1)$ with restricted information set has only one possible path before expansion and player type $(1,2)$ has two paths available to them, without loss of generality, $S_{(1,1)}=\{s_a\}$ and $S_{(1,2)}=\{s_a, s_b\}$.
		
		For any two distinct $i,j \in N$, either $E_i \cap E_j = \emptyset$ or $E_i \cap E_j = \mathcal{C}$. If $E_1 \cap E_j = \emptyset$ then we do not need to consider population $j$ as it will not affect the equilibrium costs of $(1,1)$. So assume that for all $j \in N$ we have $E_1 \cap E_j = \mathcal{C}$. Now suppose that type $k_j \in K_j$ only has one choice of path. Then we can consider an equivalent game where the costs of resources $e \in E_{(j,k_j)}$ are increased to $c_e(f_e(\bm{x})+d_{jk_j})$. So we can assume that for any $j \in N \backslash \{1\}$, there exists only one information type $(j,2)$ where each player has full information about relevant resources.
		
		If $(1,1)$ does not choose $s_b$ after the information set expansion then the ICUE remains unchanged and we are done. So suppose we have $\tilde{\bm{x}}^{s_b}_{11} >0$ and $\tilde{\bm{x}}^{s_a}_{11} < \bm{x}^{s_a}_{11}= d_{11}$. 
		
		Suppose that $\sum_{e \in s_a} f_e(\tilde{\bm{x}}) \leq \sum_{e \in s_a} f_e(\bm{x})$. If $\tilde{\bm{x}}^{s_a}_{11}=0$, then by definition of ICUE $\sum_{e \in s_a}c_e(f_e(\tilde{\bm{x}})) \geq \sum_{e \in s_b}c_e(f_e(\tilde{\bm{x}}))$ and otherwise, if $\tilde{\bm{x}}^{s_a}_{11}>0$ then $\sum_{e \in s_a}c_e(f_e(\tilde{\bm{x}})) = \sum_{e \in s_b}c_e(f_e(\tilde{\bm{x}}))$. Then it follows that $C_{11}(\tilde{\bm{x}})=\sum_{e \in s_b}c_e(f_e(\tilde{\bm{x}})) \leq \sum_{e \in s_a} c_e(f_e(\tilde{\bm{x}})) \leq \sum_{e \in s_a}c_e( f_e(\bm{x}))=C_{11}(\bm{x})$ since $c_e$ is continuous and non-decreasing. Hence, a contradiction.
		
		So it must be the case that $\sum_{e \in s_a} f_e(\tilde{\bm{x}}) > \sum_{e \in s_a} f_e(\bm{x})$. Any population $k \in N\backslash \{1\}$ has two possible strategies that must contain elements from both $s_a$ and $s_b$. For each $k$, we have corresponding strategies $s_{k_1}$ and $s_{k_2}$. 
		
		We can divide the player types into two sets as follows $A=\{i \in N, j \in \{1,2\}: C_{ij}(\bm{\tilde{x}}) > C_{ij}(\bm{x}) \}$ and $B=\{i \in N, j \in \{1,2\}: C_{ij}(\bm{\tilde{x}}) \leq C_{ij}(\bm{x}) \}$. By the contradiction assumption, $A$ is not empty and using Proposition \ref{prop:nonempty} we know that $B$ is nonempty.
		
		All possible paths between any two nodes from the irredundant network $\hat{G}$ form the set $\mathcal{S}$. Divide this into two distinct sets as $S_A=\{s \in \mathcal{S}: C(s,\bm{\tilde{x}}) > C(s,\bm{x}) \}$ and $S_B=\{s \in \mathcal{S}: C(s,\bm{\tilde{x}}) \leq C(s,\bm{x}) \}$. Consequently, we must have that $\max_{s \in S_A}\{C(s,\bm{x})-C(s,\bm{\tilde{x}})\}<0$ and $\min_{s \in S_B}\{C(s,\bm{x})-C(s,\bm{\tilde{x}})\}\geq0$.
		
		\textbf{Claim 1:} If $i \in A$ and $s \in S_B$ then $x_i^s=0$.
		
		This follows from $C(s,\bm{x}) \geq C(s,\bm{\tilde{x}})$ by definition of $S_B$. Then by definition of ICUE, $C(s, \bm{x}) \geq C_i(\bm{\tilde{x}})$. By definition of $A$, we must have $C_i(\bm{\tilde{x}}) > C_i(\bm{x})$. Hence, $x_i^s=0$. 
		
		\textbf{Claim 2:} If $i \in B$ and $s \in S_A$ then $\tilde{x}_i^s=0$.
		
		This is true since $C(s,\bm{\tilde{x}}) > C(s,\bm{x})$ by definition of $S_B$. By definition of ICUE, $C(s, \bm{x}) \geq C_i(\bm{x})$. Finally, by definition of $A$ we must have $C_i(\bm{x}) > C_i(\bm{\tilde{x}})$. Hence, $\tilde{x}_i^s=0$.
		
		Let the demands for paths in $S_A$ and $S_B$ before and after information expansion be defined as $d_A= \sum_{s \in S_A} \sum_{i \in N} x_i^s$, $\tilde{d}_A= \sum_{s \in S_A} \sum_{i \in N} \tilde{x}_i^s$, $d_B= \sum_{s \in S_B} \sum_{i \in N} x_i^s$, $\tilde{d}_B= \sum_{s \in S_B} \sum_{i \in N} \tilde{x}_i^s$. It follows from Claims 1 and 2 that we have $\tilde{d}_A \leq d_A$ and $\tilde{d}_B \geq d_B$. Since $A$ and $B$ are nonempty, it also follows that both $S_A$ and $S_B$ are nonempty.
		
		\textbf{Claim 3:} Let $S_{\alpha},S_\beta$ be any nonempty partition of $\mathcal{S}$. If we have $\tilde{d}_\alpha \leq d_\alpha$ and $\tilde{d}_\beta \geq d_\beta$, then \[ \max_{s \in S_\alpha} \left\{ C(s,\bm{x})-C(s,\bm{\tilde{x}})\right\} \geq \min_{s \in S_\beta} \left\{ C(s,\bm{x})-C(s,\bm{\tilde{x}})\right\}.\]
		
		We will prove Claim 3 by induction on the number of edges and the number of populations. Start with the base case of a circuit with three edges and one population. All possible paths of the network are $\mathcal{S}=\{e_1,e_2,e_3,e_1e_2,e_1e_3,e_2e_3\}$. Suppose that without loss of generality the population's strategy set is $S=\{e_1e_2, e_3\}$ and their information set before expansion is $E=\{e_1,e_2\}$. We have a two-terminal game that is SLI so by Theorem \ref{theorem:acemloglu} it is immune to IBP. For any $s \in \mathcal{S}$ such that $s \notin S$, $s$ contributes no demand to $d_\alpha,\tilde{d}_\alpha,d_\beta,\tilde{d}_\beta$ hence any such $s$ can be randomly assigned to one of the sets $S_\alpha$ and $S_\alpha$. Since $e_1e_2 \in E$, the demand for this strategy can only reduce after the information expansion hence $e_1e_2 \in S_\alpha$. The demand for $e_3$ can only increase after the information expansion so $e_3 \in S_\beta$. Since there is no IBP, we must have
		\[
		C(e_1e_2, \bm{x}) \geq
		\begin{dcases}
		C(e_1e_2, \bm{\tilde{x}}) & \text{if $C(e_1e_2, \bm{\tilde{x}}) \leq C(e_3, \bm{\tilde{x}})$} \\
		C(e_3, \bm{\tilde{x}}) & \text{if $C(e_1e_2, \bm{\tilde{x}}) \geq C(e_3, \bm{\tilde{x}})$}
		\end{dcases}
		\]
		If $C(e_1e_2, \bm{\tilde{x}}) \leq C(e_3, \bm{\tilde{x}})$ then $ C(e_1e_2, \bm{x})-C(e_1e_2, \bm{\tilde{x}})\geq 0$. If $C(e_1e_2, \bm{\tilde{x}}) \geq C(e_3, \bm{\tilde{x}})$, then $C(e_3, \bm{x}) \leq C(e_3, \bm{\tilde{x}})$ as $f_{e_3}(\bm{x})<f_{e_3}(\bm{\tilde{x}})$. Hence, $\max_{s \in S_\alpha} \left\{ C(s,\bm{x})-C(s,\bm{\tilde{x}})\right\} \geq \min_{s \in S_\beta} \left\{ C(s,\bm{x})-C(s,\bm{\tilde{x}})\right\}$.
		
		Now consider a subdivision of this circuit. The same reasoning holds, hence it is true for a circuit with any number of edges. Now we assume that Claim 3 is true for $n-1$ populations on a circuit with an arbitrary number of edges. Now we will try to show how the addition of another population will not affect this property. Suppose that the information of player type $(1,1)$ is the population with the expanded information set with strategy sets before $
		\tilde{S}_{11} = \{ s_\alpha, s_\beta \}$. If for any $s \in \mathcal{S}$ such that $s \notin S_{11} \cup ... \cup S_{n2}$, then $s$ contributes no demand to $d_\alpha,\tilde{d}_\alpha,d_\beta,\tilde{d}_\beta$ and hence can be arbitrarily assigned to one of the sets $S_\alpha$ and $S_\beta$.
		We can assume $C(s_\alpha,\bm{x}) \geq C(s_\beta,\bm{x})$ since otherwise $C(s,\bm{x})-C(s,\bm{\tilde{x}})=0$ $\forall s \in \mathcal{S}$ and Claim 3 is immediately true. Hence, we know that the demand for $s_\alpha$ will strictly reduce and the demand for $s_\beta$ will strictly increase. 
		
		For any $i \in \{2,...,n\}$ let the strategy set be $S_i=\{s_{\alpha i},s_{\beta i}\}$. It must be the case that $\forall i \in \{2,...,n\}$ if demand for $s_{\alpha i}$ increases then the demand for $s_{\beta i}$ must reduce. Suppose that $\exists i \in \{2,...,n\}$ such that the demands for $s_{\alpha i},s_{\beta i}$ remain the same, then we can play an equivalent game of $n-1$ populations, for which we assumed the claim to be true. So in each population, we must have a strict increase in demand for one strategy and strict decrease for the other. Since we assumed $\tilde{d}_\alpha \leq d_\alpha$ and $\tilde{d}_\beta \geq d_\beta$ and that both $S_\alpha$ and $S_\beta$ are nonempty, in every possible allocation of strategies to $S_\alpha$ and $S_\beta$ there exists $i \in \{1,...,n\}$ such that exactly one strategy of $S_i$ belongs to $S_\alpha$. Hence, there always exists a population $i \in \{1,...,n\}$ whose strategies belong in both $S_\alpha$ and $S_\beta$. Without loss of generality, let the demand for $s_{\alpha i}$ increase and belong to $S_\alpha$ and demand for $s_{\beta i}$ reduce and belong to $S_\beta$. Then we must have that $C_i(s_{\alpha i }, \bm{\tilde{x}}) \leq C_i(s_{\beta i }, \bm{\tilde{x}})$ and $C_i(s_{\alpha i }, \bm{x}) \geq C_i(s_{\beta i }, \bm{x})$. Therefore we have $C_i(s_{\alpha i }, \bm{x}) - C_i(s_{\alpha i }, \bm{\tilde{x}}) \geq C_i(s_{\beta i }, \bm{x}) - C_i(s_{\beta i }, \bm{\tilde{x}})$. This concludes the induction step and so we have proved Claim 3.
		
Using the partition of $S_A$ and $S_B$ and the three claims, we have the contradiction:  
	\[ 0 > \max_{s \in S_A} \left\{ C(s,\bm{x})- C(s,\bm{\tilde{x}})\right\} \geq \min_{s \in S_B} \left\{ C(s,\bm{x})-C(s,\bm{\tilde{x}})\right\} \geq 0 \]
	Therefore, we have reached a contradiction.
	\end{proof}

An interesting application of this theorem is considering pedestrian exit-routing in sports stadiums, such as in Figure \ref{fig:stadium}. Suppose that player populations are those who are seated in the same block and wish to use the same mode of transport. Knowledge of the layout can differ between people and can be distributed through signs or movement restrictions provided by the stadium. If pedestrians are only allowed to exit their seat block through a single exit, as in Figure \ref{fig:stadium}, then the overlapping network of all populations is a ring. Hence, information cannot harm the expected travel times of exiting the stadium. However, if visitors are allowed to walk between seat blocks before exiting the stadium the network no longer has immunity to IBP. These results should, therefore, have importance in planning evacuation routes. 

	\begin{center}
		\begin{figure} [t]
			\begin{center}
				\includegraphics[width=0.4\textwidth,trim={0 17cm 0cm 1.8cm},clip]{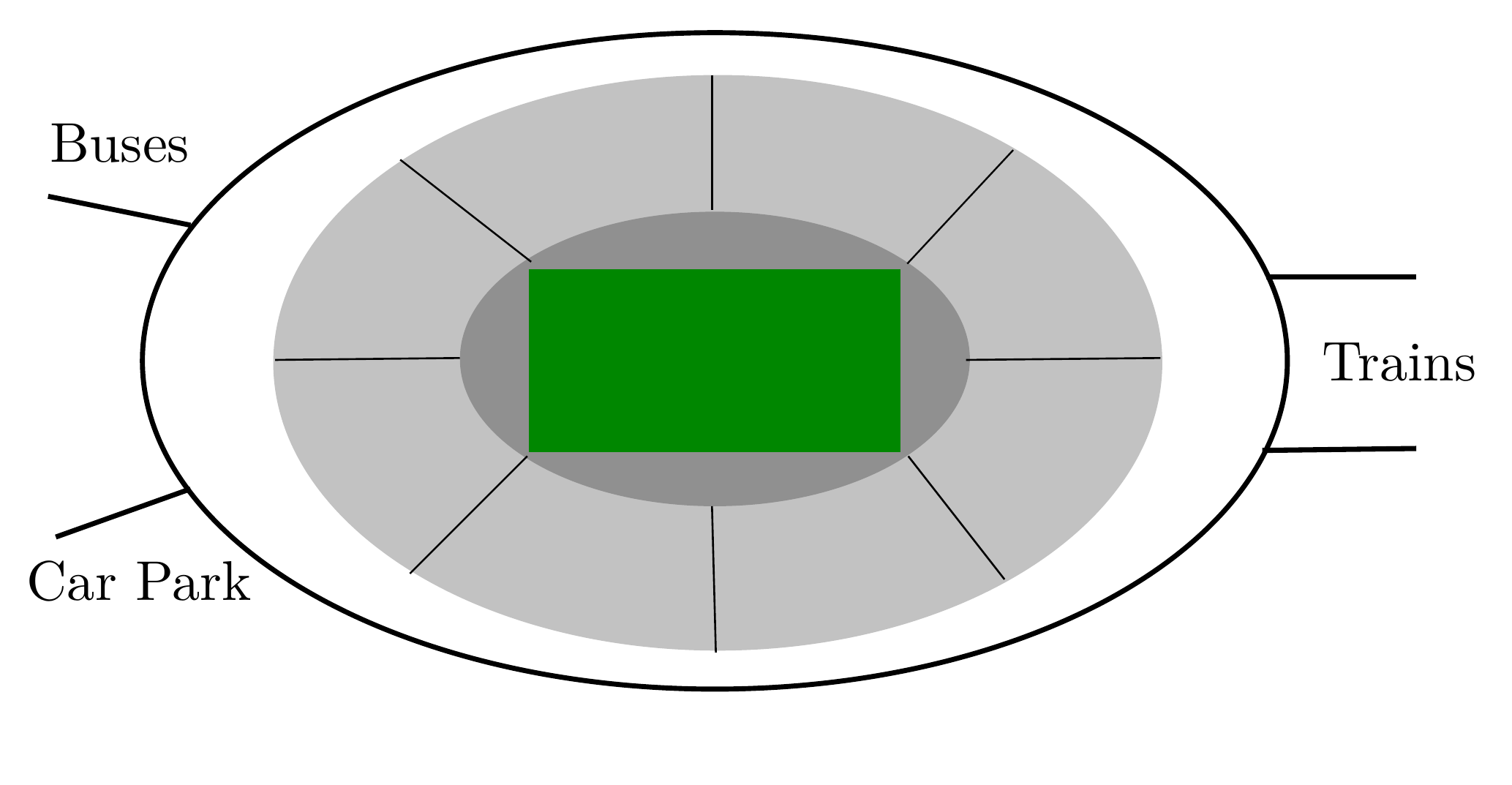}
				\hspace{0.2cm}
				\begin{tikzpicture}[shorten >=2pt, thick]
				\node[circle,draw] (bus) at (-0.8,0.8) {};
				\node[circle,draw] (car) at (-0.8,-0.8) {};
				\node[circle,draw] (A) at (0,0.4) {};
				\node[circle,draw] (B) at (1,1) {};
				\node[circle,draw] (C) at (2,1) {};
				\node[circle,draw] (D) at (3,0.4) {};
				\node[circle,draw] (a) at (0.6,0.2) {};
				\node[circle,draw] (b) at (1.2,0.5) {};
				\node[circle,draw] (c) at (1.8,0.5) {};
				\node[circle,draw] (d) at (2.4,0.2) {};
				\node[circle,draw] (train) at (3.6,0.4) {};
				\node[circle,draw] (trains) at (3.6,-0.4) {};
				\node[circle,draw] (E) at (3,-0.4) {};
				\node[circle,draw] (F) at (2,-1) {};
				\node[circle,draw] (G) at (1,-1) {};
				\node[circle,draw] (H) at (0,-0.4) {};
				\node[circle,draw] (e) at (2.4,-0.2) {};
				\node[circle,draw] (f) at (1.8,-0.5) {};
				\node[circle,draw] (g) at (1.2,-0.5) {};
				\node[circle,draw] (h) at (0.6,-0.2) {};
				\draw[-] (A) to[bend left=0] node[above] {} (bus);
				\draw[-] (bus) to[bend right=0] node[above] {} (A);
				\draw[-] (H) to[bend left=0] node[above] {} (car);
				\draw[-] (car) to[bend right=0] node[above] {} (H);
				\draw[-] (A) to[bend left=20] node[above] {} (B);
				\draw[-] (B) to[bend right=20] node[above] {} (A);
				\draw[-] (B) to[bend left=0] node[above] {} (C);
				\draw[-] (C) to[bend right=0] node[above] {} (B);
				\draw[-] (C) to[bend left=20] node[below] {} (D);
				\draw[-] (D) to[bend right=20] node[below] {} (C);
				\draw[-] (D) to[bend left=0] node[below] {} (E);
				\draw[-] (E) to[bend right=0] node[below] {} (D);
				\draw[-] (D) to[bend left=0] node[below] {} (train);
				\draw[-] (train) to[bend right=0] node[below] {} (D);
				\draw[-] (trains) to[bend left=0] node[below] {} (E);
				\draw[-] (E) to[bend right=0] node[below] {} (trains);
				\draw[-] (E) to[bend left=20] node[below] {} (F);
				\draw[-] (F) to[bend right=20] node[below] {} (E);
				\draw[-] (F) to[bend left=0] node[above] {} (G);
				\draw[-] (G) to[bend right=0] node[above] {} (F);
				\draw[-] (H) to[bend right=0] node[below] {} (A);
				\draw[-] (A) to[bend left=0] node[below] {} (H);
				\draw[-] (H) to[bend right=20] node[below] {} (G);
				\draw[-] (G) to[bend left=20] node[below] {} (H);
				\draw[-] (a) to[bend left=0] node[above] {} (A);
				\draw[-] (A) to[bend left=0] node[above] {} (a);
				\draw[-] (b) to[bend left=0] node[above] {} (B);
				\draw[-] (B) to[bend left=0] node[above] {} (b);
				\draw[-] (c) to[bend left=0] node[above] {} (C);
				\draw[-] (C) to[bend left=0] node[above] {} (c);
				\draw[-] (d) to[bend left=0] node[above] {} (D);
				\draw[-] (D) to[bend left=0] node[above] {} (d);
				\draw[-] (e) to[bend left=0] node[above] {} (E);
				\draw[-] (E) to[bend left=0] node[above] {} (e);
				\draw[-] (f) to[bend left=0] node[above] {} (F);
				\draw[-] (F) to[bend left=0] node[above] {} (f);
				\draw[-] (g) to[bend left=0] node[above] {} (G);
				\draw[-] (G) to[bend left=0] node[above] {} (g);
				\draw[-] (h) to[bend left=0] node[above] {} (H);
				\draw[-] (H) to[bend left=0] node[above] {} (h);
				\end{tikzpicture}
				\caption{Evacuation of a sports stadium is immune to IBP.}
				\label{fig:stadium}
			\end{center}
		\end{figure}
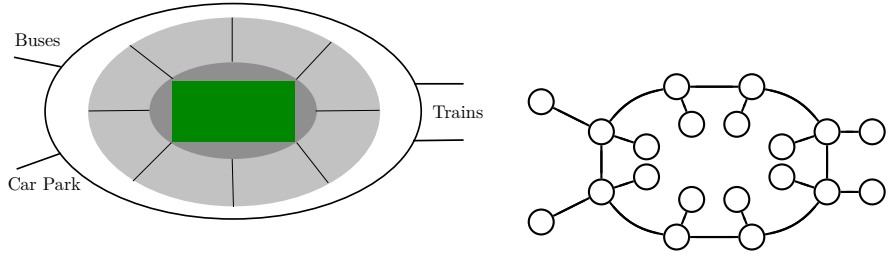
	\end{center}
	
	 \section{IBP for Social Cost} \label{sec:ibpsc}
	  Thus far, we define IBP to be the comparison between equilibrium costs of a player whose information set is expanded. A natural weakness of using IBP to analyse the system as a whole is that it does not incorporate any effects that the information expansion has on all players. From a mechanism design perspective, it is more relevant to compare the social costs of the ICUEs. In this section, we define our own version of a paradox, similar to that of IBP, and show that its occurrence is independent of network topology.
	
	 Define \textit{informational Braess' paradox for the social cost} (IBPSC) as the case when one type of player has their information set strictly expanded which causes an increase to the social cost. IBP is a specific case of IBPSC that occurs when information harms the informed player.
	
	 In order to consider how the social cost will change in an ICUE equilibrium, we need the following proposition for the simplest circuit network: a two-node, two-edge ring most commonly referred to as the Pigou network.
	 	\begin{center}
		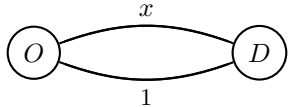
\begin{figure} [h]
			\begin{center}
				\begin{tikzpicture}[shorten >=2pt, thick]
				\node[circle,draw] (A) at (0,0) {$O$};
				\node[circle,draw] (D) at (3,0) {$D$};
				\draw[-] (A) to[bend left=20] node[above] {$x$} (D);
				\draw[-] (D) to[bend right=20] node[above] {} (A);
				\draw[-] (A) to[bend left=-20] node[below] {\small $1$} (D);
				\draw[-] (D) to[bend right=-20] node[below] {} (A);
				\end{tikzpicture}
				\end{center}
				\caption{A Pigou network.}
				\label{fig:Pigou}
				\end{figure}
				\end{center}
				
	 \begin{prop} \label{prop:pigou}
	 Any two terminal network with at least two distinct paths has the Pigou network embedded in it. 
	 \end{prop}
	 \begin{proof}
	 Suppose the statement is not true, then there does exist a network $G$ with at least two distinct paths which does not embed the Pigou network. Since a 2-edge ring is the simplest network that contains a cycle, any cyclic network must embed the 2-edge ring. Therefore, $G$ is acyclic. Any acyclic undirected connected network must be a tree. In a tree, the number of possible paths between any two pairs of nodes is necessarily one. This contradicts the original statement.
	 \end{proof}
Now we show that there will always be a set of cost functions for any network that allows for IBPSC to occur.
	 \begin{theorem}
	 There exists an assignment of information sets and cost functions on any two-terminal network such that IBPSC occurs.
	 \end{theorem}
 
\begin{proof}
In order for a network to have an IBPSC, at least one type of player must have an information set as a strict subset of their resources. Therefore, at least two distinct strategies must exist. If there exists two distinct strategies then, by Proposition \ref{prop:pigou}, the network must have the Pigou network embedded in it. Therefore, it suffices to show that there always exists an assignment of information sets and cost functions on the Pigou network such that IBPSC exists. To show this, we will use Pigou's example as shown in Figure \ref{fig:Pigou}.

Consider two populations with demands $d_1=1, \, d_2=1$. Suppose their information sets are $E_1=\{e_1\}$ and $E_2=\{e_1, \, e_2\}=E$. Let $c_{e_1}(\bm{x})= x$ and $c_{e_2}(\bm{x})= 2$. At equilibrium, type 1 will choose $e_1$ and type 2 will choose $e_1$ with a social cost of $3$. Now, if we expand type 1's information set to be $E_1=\{e_1, \, e_2\}=E$, then since $e_1$ costs strictly less than $e_2$, type 1 will switch their strategy and the new equilibrium has social cost of 4.
\end{proof}
 
	 Despite its similarity with BP, the occurrence of IBPSC is actually independent of the topology of the network. The IBPSC occurs because of the difference between the socially optimal cost and the cost of a UE. By having the ability to assign information sets, an ICUE outcome can be found with a cost strictly less than any UE. Any ICUE that is not a UE must be unstable as at least one player wishes to deviate from it if they have their information set expanded. There will always exist an assignment of information sets for any network where the social cost expands as information sets expand.
	
	 \begin{theorem}
	 Immunity to IBPSC is independent of network topology.
	 \end{theorem}	
	 Since no two terminal network is immune to IBPSC, it is simple to extend the same observation for multiple OD pairs. This result shows that there will always be an assignment of information sets and cost functions that can create IBPSC.
	
	\section{Concluding Remarks}
	
In this paper, we have looked at the selfish routing games where individuals do not have perfect knowledge of the network structure and how increasing their knowledge can affect the overall performance.	Specifically, we have identified a natural class of networks, i.e., rings, that are immune to performance deterioration, which is known in the literature as informational Braess' Paradox. This has settled a conjecture in \cite{Acemoglu}. We have also shown that under simple alternative definition of performance, which is in line with the complete information variant of Bress' Paradox, all networks are resistant to such phenomenon.

We believe that the identification of safe network structure is bound to fundamentally impact the design of transportation networks.
A similar argument can be made for the simulation of traffic flow, as different network structure will give rise to different behaviour by boundedly rational participants.

\bibliographystyle{abbrv}
\bibliography{bibliography3}

\end{document}